\documentclass{amsart}
\usepackage{amsthm}
\pdfoutput=1
\usepackage{amssymb}
\usepackage{amsfonts}
\usepackage{mathrsfs}
\usepackage[all,cmtip]{xy}
\usepackage{tikz}     
\usepackage{multirow}
\usepackage{url}

\begin{document}

\date{}
\title[The very good property for parabolic bundles]{The very good property for parabolic vector bundles over curves}
\author{Alexander Soibelman}
\address{Max Planck Institute for Mathematics, Vivatsgasse 7,
53111 Bonn, Germany}
\email{asoibel@mpim-bonn.mpg.de}

\newtheorem{thm}{Theorem}[section]
\newtheorem{defn}[thm]{Definition}
\newtheorem{lmm}[thm]{Lemma}
\newtheorem{prp}[thm]{Proposition}
\newtheorem{conj}[thm]{Conjecture}
\newtheorem{exa}[thm]{Example}
\newtheorem{cor}[thm]{Corollary}
\newtheorem{que}[thm]{Question}
\newtheorem{ack}[thm]{Acknowledgments}
\newtheorem{clm}[thm]{Claim}
\newtheorem{rmk}[thm]{Remark}

\newcommand{\dimn}{\operatorname{dim}}

\begin{abstract}
The purpose of this note is to extend Beilinson and Drinfeld's ``very good" property to moduli stacks of parabolic vector bundles on curves of genuses $g = 0$ and $g = 1$.  Beilinson and Drinfeld show that for $g > 1$ a trivial parabolic structure is sufficient for the moduli stacks to be ``very good".  We give a necessary and sufficient condition on the parabolic structure for this property to hold in the case of lower genus.     
\end{abstract}

\maketitle

\section{Introduction}

In \cite{BD1991} Beilinson and Drinfeld introduced the ``good'' and ``very good'' properties for algebraic stacks to avoid derived categories in their study of $D$-modules on the moduli stack of $G$-bundles $\textrm{Bun}_G(X)$ on a curve $X$.  An equidimensional algebraic stack $\mathcal{Y}$ is \textit{good} if 
$$
\textrm{codim } \{y \in \mathcal{Y}|\textrm{Aut}(y) = n \} \ge n \quad \forall n > 0.
$$ 
It is \textit{very good} if the inequality on the codimension is strict.  Both the good and the very good properties have certain consequences for the geometry of the cotangent stack associated to $\mathcal{Y}$.  

Namely, the good property implies there is a reasonable definition of the cotangent stack $T^*\mathcal{Y}$, without passing to derived objects (see e.g. \cite{Ra2009} for detailed discussion), while the very good property also provides an open, dense Deligne-Mumford substack in $T^*\mathcal{Y}$ coming from the points with dimension $0$ automorphism groups.

Beilinson and Drinfeld prove that $\textrm{Bun}_G(X)$ is very good for semisimple, complex $G$ and smooth, connected, projective $X$ of genus $g > 1$.  In the lower genus cases, this is no longer true.  One possible way to get around this is to introduce additional structure to ``decrease'' the number of points with nontrivial automorphisms. 

In this paper, we will be dealing with the case of vector bundles (i.e. when $G = GL(n)$).  In order to make the corresponding moduli stacks good or very good we will consider vector bundles together with additional \textit{parabolic structure}.  Defined by Mehta and Seshadri in \cite{MS1980}, a \textit{parabolic bundle} $\mathbf{E}$ is a vector bundle $E$ together with a collection of partial flags $E_{x_i} \supset E_{i1} \supset \cdots \supset E_{iw_i} = 0$ in the fibers over points in the reduced effective divisor $D = x_1 + \cdots + x_k$ (note that some texts call this a \textit{quasi-parabolic bundle}).  

Let $w = (w_1, \dots, w_k)$. We call $(D,w)$ the \textit{weight type} of $\mathbf{E}$.  If $\alpha_ 0 = \textrm{rk } E$ and $\alpha_{ij}  = \dimn E_{ij}$, then $\alpha = (\alpha_0, \alpha_{ij})$ is called the \textit{dimension vector} of $\mathbf{E}$. 

Note that by analogy with vector bundles, one can define morphisms between parabolic bundles $\mathbf{E}$ and $\mathbf{F}$ of the same weight type, as well as the locally free sheaf $\mathscr{H}om_{\textrm{Par}}(\mathbf{E}, \mathbf{F})$ of parabolic bundle morphisms.

We will need the bilinear \textit{Euler form}, which may be expressed as follows on dimension vectors:
$$
\langle \alpha, \beta  \rangle = \alpha_0\beta_0 + \sum_{1 \le i \le k,\\ 0\le j \le w_i - 1} (\beta_{ij+1} - \beta_{ij}) \alpha_{ij+1},
$$
where $\alpha_{i0} = \alpha_0$ and $\beta_{i0} = \beta_0$.  We denote $q(\alpha) := \langle \alpha, \alpha \rangle$ and $p(\alpha) = 1 - q(\alpha)$.

In \cite{So2013} we provided sufficient conditions for the moduli stack of parabolic bundles over $\mathbb{P}^1$ to be very good and showed this has interesting applications to the additive and multiplicative Deligne-Simpson problems (see e.g. \cite{CB2003}, \cite{CBS2006} \cite{Ko22004}, \cite{Si1991}).  The purpose of this note is to formulate necessary and sufficient conditions on the dimension vector $\alpha$ for the corresponding moduli stack of parabolic bundles of weight type $(D,w)$ to be good or very good over curves of arbitrary genus.  

Note that $GL(n)$ has a $1$-dimensional central subgroup, which acts by dilation on the fibers of a vector bundle and preserves any flag in that fiber. Therefore, a parabolic bundle always has an automorphism group of at least dimension $1$.  It follows that the moduli stack of parabolic vector bundles cannot be very good according to the original definition.  Therefore we modify it in the next section, to fit this case (similarly done in \cite{So2013}).  

There are two further sections.  In the third section, we formulate and prove the main result concerning the very good property for the moduli stack of parabolic bundles over a curve.  The proofs in Section 3 are based on some technical results we moved to Section 4.  

The contents of the last two sections are largely contained in \cite{So2013}, though we include all of the details from there for the sake of completeness.  However, unlike in \cite{So2013}, computations are done for smooth projective curves of arbitrary genus, and the conditions for the good or very good properties to hold are both necessary and sufficient. 

Note that for the case of genus $g > 1$ and $G = \textrm{GL}(n)$ our results provide an alternative to Beilinson and Drinfeld's proof of the very good property that does not require showing the global nilpotent cone of the Hitchin system is Lagrangian (see e.g. \cite{La1988} for an analogous proof of this).  In the $g = 1$ case, even two points with (arbitrary) nontrivial parabolic structure are sufficient for the corresponding moduli stack to be very good, in our formulation.

\section{The Almost Very Good property}

Since the automorphism group for any vector bundle has a $1$-dimensional subgroup, which preserves flag structure, the default definition of the very good property does not hold for the moduli stack of parabolic bundles on a curve.  Instead, we modify the definition of good and very good stacks, in order to account for this subgroup.  Such a modification is harmless because.

Let $K$ be an algebraically closed field, and let $\mathcal{Y}$ be an equidimensional algebraic stack over $K$. Let $\textrm{Aut}(y)$ the automorphism group of a point $y \in \mathcal{Y}$.  Denote by $\mathcal{Y}^n = \{y \in \mathcal{Y}| \textrm{dim Aut}(y) = n\}$ the reduced locally closed substack of $\mathcal{Y}$.

We say $\mathcal{Y}$ is \textit{almost good} if 
$$
\textrm{codim } \mathcal{Y}^{n + m} \ge n,
$$
for $m = \min_y \textrm{dim Aut}(y)$ and any $n > 0$.  If the inequality is strict, then we say $\mathcal{Y}$ is \textit{almost very good}.

Let $\mathcal{I}_{\mathcal{Y}}$ be the inertia stack associated with $\mathcal{Y}$.  Note that we have the decomposition
$$
\mathcal{I}_{\mathcal{Y}} = \coprod_{i} \mathcal{I}^i,
$$
where $\mathcal{I}^i$ is the locally closed, reduced substack corresponding to points of $\mathcal{Y}$ with $\textrm{min dim Aut}(y) = i$.  We will need the following theorem:
\begin{thm}
\label{2.1}
The stack $\mathcal{Y}$ is almost good if and only if $\dimn \coprod_{i} \mathcal{I}^{i>m} - m \le \dimn  \mathcal{Y}$.  If the inequality if strict, then it is almost very good.
\end{thm}
\begin{proof}
It is easy to see that $\dimn \mathcal{I}^i =  \dimn \mathcal{Y}^i + i$.  Therefore, we have
$$
\textrm{codim } \mathcal{Y}^{n + m} = \dimn \mathcal{Y} - \dimn \mathcal{Y}^{n+m} = \dimn 
\mathcal{Y} - \dimn \mathcal{I}^{n + m} + (n+m), 
$$  
for all $n > 0$.  Thus the almost good property is equivalent to 
$$
\dimn \mathcal{Y} \ge \dimn \mathcal{I}^{n+m} - m,
$$
for all $n > 0$.  The theorem statement follows.  Making the corresponding inequality strict, we also obtain the statement about the very good property.
\end{proof}

\section{The Almost Very Good property for Parabolic Bundles}

As before, let $X$ be a smooth projective curve of genus $g$.  As in the Introduction, we will fix parabolic weight type $(D,w)$ and dimension vector $\alpha = (\alpha_0, \alpha_{ij})$.  

Denote by $\mathcal{B} := \textrm{Bun}_{D,w,\alpha}(X)$ the moduli stack of parabolic bundles of weight type $(D,w)$ and dimension vector $\alpha$ on $X$.  Let $\mathcal{I}_{\mathcal{B}}$ be the corresponding inertia stack.  Furthermore, let $\mathcal{P}_\mathcal{B}:= \mathcal{P}(D,w,\alpha)$ be the moduli stack parametrizing pairs $(\mathbf{E}, f)$, where $\mathbf{E} \in \textrm{Bun}_{D,w,\alpha}(X)$ and $f \in \textrm{End}_{\textrm{Par}}(\mathbf{E})$.  Finally, let $\mathcal{N} (D,w,\alpha)$ be the closed, reduced substack of $\mathcal{P}_{\mathcal{B}}$ for which $f$ is nilpotent.  

The stacks $\textrm{Bun}_{D,w,\alpha}(X)$, $\mathcal{I}_\mathcal{B}$, $\mathcal{P}_\mathcal{B}$, and $\mathcal{N}(D,w,\alpha)$ all have standard presentations as algebraic stacks.  Moreover, $\textrm{Bun}_{D,w,\alpha}(X)$ is smooth.

By the Riemann-Roch theorem, we have for $\mathbf{E}$ and $\mathbf{F} \in \textrm{Bun}_{D,w,\alpha}(X)$: 
\begin{align*}
\chi(\mathscr{H}om_{\textrm{Par}}(\mathbf{F}, \mathbf{E})) & = \textrm{deg}(\mathscr{H}om_{\textrm{Par}}(\mathbf{F}, \mathbf{E})) + (1-g)\alpha_0\beta_0 \\ & = \beta_0 \cdot \textrm{deg}(E) - \alpha_0 \cdot \textrm{deg}(F) - g\alpha_0\beta_0 + \langle \beta, \alpha  \rangle,
\end{align*}
where $\langle \beta, \alpha \rangle$ is the Euler form, mentioned in the introduction.  Consequently, we have that 
$$
\dimn \textrm{Bun}_{D,w,\alpha}(X) =  - \chi(\mathscr{E}nd_{\textrm{Par}}(\mathbf{E})) = g \cdot \alpha_0^2 - q(\alpha).
$$

Note that in the case of $\textrm{Bun}_{D,w,\alpha}(X)$ we have $\dimn \textrm{End}_{\textrm{Par}}(\mathbf{E}) = \dimn \textrm{Aut}(\mathbf{E})$, therefore we can replace the inertia stack with $\mathcal{P}_{\mathcal{B}}$ in Theorem \ref{2.1}.  This lets us reduce the necessary computations for $\mathcal{I}_{\mathcal{B}}$ to computations for $\mathcal{N}(D,w,\alpha)$.  Specifically, we have the following lemma:
\begin{lmm}
\label{3.1}
There exists a decomposition into nonnegative dimension vectors $\alpha = \sum_{l=1}^r \gamma^{(l)}$ such that $\dimn \mathcal{P}_{\mathcal{B}} = r + \sum_{l=1}^r \dimn \mathcal{N}(D,w,\gamma^{(l)})$.
\end{lmm}
\begin{proof}
Fix a point of $x \in \mathbb{A}^{\alpha_0}$.  This defines a characteristic polynomial $x(t)= (t -  \lambda_1)^{m_1}(t - \lambda_2)^{m_2} \cdots (t - \lambda_r)^{m_r}$, where $\lambda_i \neq \lambda_j$ for $i \neq j$ (notice that the eigenvalues do not depend on the point of $X$).  Let $c: \mathcal{P}_{\mathcal{B}} \rightarrow \mathbb{A}^{\alpha_0}$ be the morphism defined by sending the pair $(\mathbf{E}, f)$ to the coefficients of the characteristic polynomial $\textrm{char}(f)$ of $f$.  Consider $(\mathcal{P}_{\mathcal{B}})_x$, the fiber of $c$ over $x$.  The points of $(\mathcal{P}_{\mathcal{B}})_x$
may be identified with pairs $(\mathbf{E}, f)$, such that $f$ is an endomorphism of the parabolic bundle
$\mathbf{E}$ with $\textrm{char}(f) = x(t)$.  Therefore, $\mathbf{E}$ decomposes as 
$$
\mathbf{E} = \bigoplus_l \textbf{ker}(f-\lambda_l)^{m_i}.
$$
Let $(\mathcal{P}_{\mathcal{B}})_x$ be the fiber of $c$ over $x$.  We have $ \dimn (\mathcal{P}_{\mathcal{B}})_x =  \dimn \prod_l \mathcal{P}_l$,  where 
$\mathcal{P}_l$ is the stack of pairs $(\mathbf{E}_l, f_l)$ such that $\mathbf{E}_l$ is a parabolic bundle and $f_l$ is its endomorphism satisfying $\textrm{char}(f_l) = (t - \lambda_l)^{m_l}$.  Since $f_l - \lambda_l$
is nilpotent, we can compute 
$$\dimn \mathcal{P}_l = \textrm{dim } \mathcal{N}(D,w,\gamma^{(l)}),$$
for some dimension vector $\gamma^{(l)} \le \alpha$.  Note that $\alpha = \gamma^{(1)} + \cdots + \gamma^{(r)}$.  Since $c$ maps $(\mathcal{P}_{\mathcal{B}})_x$ to the subvariety consisting of monic polynomials with $r$ distinct roots, we can compute:
$$
\dimn \mathcal{P}_{\mathcal{B}} = r + \sum_{l=1}^r \dimn \mathcal{N}(D,w,\gamma^{(l)}),
$$
for some decomposition $\alpha = \sum_{l=1}^r \gamma^{(l)}$ into nonnegative dimension vectors. 
\end{proof}

We now wish to compute the dimension of $\mathcal{N}(D,w,\alpha)$.
\begin{thm}
\label{3.2}
There exists a decomposition into nonnegative dimension vectors $\alpha = \sum_{l=1}^s \beta^{(l)}$ such that $\dimn \mathcal{N}(D,w,\alpha) = g \cdot \sum_{l = 1}^s (\beta^{(i)}_0)^2 - \sum_{l = 1}^s q(\beta^{(i)})$.
\end{thm}
\begin{proof}
Let $(\mathbf{E}, f)$ be a point of $\mathcal{N}(D,w, \alpha)$.  Let $\mathbf{F} = \mathbf{ker} \ f$ and $\mathbf{H} = \mathbf{E}/ \mathbf{F}$.  We will proceed by induction on the rank of the vector bundle $E$ (note that this is $\alpha_0$ in our notation). 

Define $\mathcal{N}_{\beta}(D,w, \alpha)$ to be the substack of $\mathcal{N}(D,w, \alpha)$ consisting of objects $(\mathbf{E},f)$ such that the corresponding $\mathbf{F}$ belongs to $\textrm{Bun}_{D,w, \beta}(X)$.  Choose $\beta$ such that
$$
\dimn \mathcal{N}(D,w, \alpha) = \dimn \mathcal{N}_{\beta}(D,w, \alpha).
$$
Consider the morphism
$$
\phi: \mathcal{N}_{\beta}(D,w, \alpha) \rightarrow \mathcal{N}(D,w, \alpha - \beta),
$$
which is defined by sending $(\mathbf{E}, f)$ to
$(\mathbf{H}, f|_{\mathbf{H}})$ $\in \mathcal{N}(D, w, \alpha - \beta)$, with corresponding restrictions on the arrows.  In this case, we get by induction that
\begin{align*}
& \textrm{dim } \mathcal{N}_{\beta}(D,w, \alpha) = \textrm{dim } \mathcal{N}_{\beta}(D,w, \alpha)_x + \dimn \mathcal{N}(D,w, \alpha -\beta)\\ &= \textrm{dim } \mathcal{N}_{\beta}(D,w, \alpha)_x + g \cdot \sum_i (\beta^{(i)}_0)^2 - \sum_i q(\beta^{(i)}),
\end{align*}
for some $x = (\mathbf{H},h) \in \mathcal{N}(D, w, \alpha - \beta)$ and $\alpha - \beta = \sum_i \beta^{(i)}$.  Now, we wish to compute the dimension of the fiber $\mathcal{X} = \mathcal{N}_{\beta}(D,w, \alpha)_x$ of $\phi$.  

Let $\mathcal{P}_{\mathbf{V}} := \mathcal{P}_{\mathbf{V}}(D,w,\alpha)$ be the algebraic stack consisting of pairs $\{\mathbf{W}, i: \mathbf{V} \hookrightarrow \mathbf{W} \}$, where $i$ is an inclusion of parabolic bundles and $\mathbf{W}$ is a parabolic bundle of weight type $(D,w)$ and dimension vector $\alpha$.  Let $\mathbf{F_1} = \mathbf{ker} \ h$ and let $\mathcal{X}' = \mathcal{P}_{\mathbf{F_1}}(D, w, \beta)$.  We have two morphisms $\psi_1: \mathcal{X} \rightarrow \textrm{Bun}_{D,w, \beta}(X)$ and $\psi_2: \mathcal{X}' \rightarrow \textrm{Bun}_{D,w, \beta}(X)$, where $\psi_1$ sends the pair $(\mathbf{E},f)$ to $\textbf{ker } f$ and likewise $\psi_2$ sends $(\mathbf{F},i)$ to $\mathbf{F}$. 

The deformations of elements of the fiber $\mathcal{X}_{\mathbf{F}}$ are governed by the hypercohomology of the complex 
$$
\mathscr{H}om_{\textrm{Par}}(\mathbf{H}, \mathbf{F}) \xrightarrow{h} \mathscr{H}om_{\textrm{Par}}(\mathbf{H}, \mathbf{F}),
$$
defined in Lemma \ref{4.3.4}.  Therefore, by Lemma \ref{4.3.4}, we get that:
$$
\textrm{dim } \mathcal{X}_{\mathbf{F}} = \textrm{dim } H^1 (X, \mathscr{H}om_{\textrm{Par}}(\mathbf{F_1}, \mathbf{F})).
$$

Furthermore, since $f$ induces the injective morphism $\textbf{ker } f^2/\textbf{ker }f \rightarrow \textbf{ker } f$, then the fiber 
$\mathcal{X}'_{\mathbf{F}}$ is nonempty.  Therefore, 
$$
\textrm{dim } \mathcal{X}'_{\mathbf{F}} = \textrm{dim } H^0 (X, \mathscr{H}om_{\textrm{Par}}(\mathbf{F_1}, \mathbf{F})).
$$

Thus, $\textrm{dim } \mathcal{X}_{\mathbf{F}}  = \textrm{dim } \mathcal{X}'_{\mathbf{F}} - \chi(\mathscr{H}om_{\textrm{Par}}(\mathbf{F_1}, \mathbf{F}))$.  We have $\textrm{dim } \mathcal{X} = \textrm{dim } \mathcal{X}' - \chi(\mathscr{H}om_{\textrm{Par}}(\mathbf{F_1}, \mathbf{F}))$. So, we obtain
$$
\textrm{dim } \mathcal{N}_{\beta}(D,w, \alpha) \le \textrm{dim } \mathcal{X}'  - \chi(\mathscr{H}om_{\textrm{Par}}(\mathbf{F_1}, \mathbf{F})) + g \cdot \sum_i (\beta^{(i)}_0)^2 - \sum_i q(\beta^{(i)}).
$$
It follows from Lemma \ref{4.3.3} that $\dimn \mathcal{X}' = \chi(\mathscr{H}om_{\textrm{Par}}(\mathbf{F_1}, \mathbf{F})) -  \chi (\mathscr{E}nd_{\textrm{Par}}(\mathbf{F}))$, which means
$$
\textrm{dim } \mathcal{N}_{\beta}(D,w, \alpha) = -\chi (\mathscr{E}nd_{\textrm{Par}}(\mathbf{F})) + g \cdot \sum_i (\beta^{(i)}_0)^2 - \sum_i q(\beta^{(i)}).
$$
Since $\chi (\mathscr{E}nd_{\textrm{Par}}(\mathbf{F})) = - g \cdot \beta_0^2 + q(\beta)$ we obtain
$$
\textrm{dim } \mathcal{N}_{\beta}(D,w, \alpha) = g (\beta_0^2 + \sum_i (\beta^{(i)}_0)^2) - q(\beta) - \sum_i q(\beta^{(i)}).
$$
The result follows.
\end{proof}

The following corollary follows from the proof of Theorem \ref{3.2}.
\begin{cor}
\label{3.3}
There exists a decomposition into at least $2$ nonnegative dimension vectors $\alpha = \sum_{l=1}^s \beta^{(l)}$ such that $\dimn (\mathcal{N}(D,w,\alpha) - \mathcal{N}_{\alpha}(D,w,\alpha)) = g \cdot \sum_{l = 1}^s (\beta^{(i)}_0)^2 - \sum_{l = 1}^s q(\beta^{(i)})$
\end{cor}

Putting together Lemma \ref{3.1} and Theorem \ref{3.2} we obtain the following corollary:
\begin{cor}
\label{3.4}
There exists a decomposition $\alpha = \sum_l \beta^{(l)}$ into positive dimension vectors such that $\dimn \mathcal{P}_{\mathcal{B}} = r + g \cdot \sum_{m = 1}^r \sum_{l = 1}^{s_m} (\beta^{(l)}_0)^2 - \sum_{m = 1}^r \sum_{l = 1}^{s_m} q(\beta^{(l)})$.
\end{cor}

Let $\mathcal{P}_{\mathcal{B}}^{i}$ be the components of $\mathcal{P}_{\mathcal{B}}$ corresponding to elements with $i$-dimensional endomorphism group.

\begin{lmm}
\label{3.5}
There exists a decomposition into nonnegative dimension vectors $\alpha = \sum_{l=1}^r \gamma^{(l)}$ such that one of the following holds:
\begin{enumerate}
\item We have $r = 1$ and $\dimn (\mathcal{P}_{\mathcal{B}} - \mathcal{P}_{\mathcal{B}}^{1}) = 1 + \dimn (\mathcal{N}(D,w,\alpha) - \mathcal{N}_{\alpha}(D,w,\alpha))$.
\item We have $r > 1$ and $\dimn (\mathcal{P}_{\mathcal{B}} - \mathcal{P}_{\mathcal{B}}^{1}) = r + \sum_{l=1}^r \dimn \mathcal{N}(D,w,\gamma^{(l)})$.
\end{enumerate}
\end{lmm}
\begin{proof}

Consider the morphism $c: \mathcal{P}_{\mathcal{B}} \rightarrow \mathbb{A}^{\alpha_0}$ in Lemma \ref{3.1} restricted to $(\mathcal{P}_{\mathcal{B}} - \mathcal{P}_{\mathcal{B}^1})$.  Following the proof of Lemma \ref{3.1}, we have 
$$
\dimn (\mathcal{P}_{\mathcal{B}} - \mathcal{P}_{\mathcal{B}^1}) = r + \dimn (\mathcal{P}_{\mathcal{B}} - \mathcal{P}_{\mathcal{B}^1})_x, 
$$
where $\dimn (\mathcal{P}_{\mathcal{B}} - \mathcal{P}_{\mathcal{B}^1})_x$ is some fiber of $c$.  If $r > 1$, we have that $\dimn \mathcal{P}_{\mathcal{B}} = \dimn (\mathcal{P}_{\mathcal{B}} - \mathcal{P}_{\mathcal{B}}^{1})$, so the result follows by Lemma \ref{3.1}.

Suppose r = 1.  From the proof of Lemma \ref{3.1}, we have:
$$ 
\dimn (\mathcal{P}_{\mathcal{B}} - \mathcal{P}_{\mathcal{B}}^1) = \dimn (\mathcal{P}_{1} - \mathcal{P}_{\mathcal{B}}^1),
$$
where $\mathcal{P}_{1}$ is the closed, reduced substack of $\mathcal{P}_{\mathcal{B}}$ corresponding to the pairs $(\mathbf{E},f)$ such that $\textrm{char}(f) = (t - \lambda)^{\alpha_0}$ for a $\lambda \in \mathbb{C}$.  We have that $\mathcal{P}_{1} = \coprod_i \mathcal{P}_1^i$, where the reduced, locally closed substack $\mathcal{P}_1^i$ corresponds to pairs $(\mathbf{E},f)$ in $\mathcal{P}_1$ such that $\mathbf{E}$ has $i$-dimensional endomorphism algebra.  Note that $\mathcal{P}_1^1 = \mathcal{P}_{\mathcal{B}}^1$.  

We also have $\mathcal{N}(D,w,\alpha) - \mathcal{N}_{\alpha}(D,w,\alpha) = \coprod_i \mathcal{N}^i$, where pairs in $\mathcal{N}^i$ correspond parabolic bundles with $i$-dimensional automorphism groups.  Note that  
$\mathcal{N}^i$ and $\mathcal{P}_1^i$ have the same image under the projections to $\textrm{Bun}_{D,w,\alpha}(X)$ for $i > 1$.  The dimensions of the fibers are, respectively, $i-1$ and $i$.  It follows that $\dimn \mathcal{P}_1^i = 1 + \dimn_i \mathcal{N}^i$.  Thus we have:
$$
\dimn (\mathcal{P}_{\mathcal{B}} - \mathcal{P}_{\mathcal{B}}^1) = \dimn (\mathcal{P}_{1} - \mathcal{P}_{\mathcal{B}}^1) = 1 + \dimn (\mathcal{N}(D,w,\alpha) - \mathcal{N}_{\alpha}(D,w,\alpha)). 
$$
The result follows.
\end{proof}

Recall that $\mathcal{I}_{\mathcal{B}}^i$ is the component of the inertia stack $\mathcal{I}_{\mathcal{B}}$ corresponding to elements with $i$-dimensional automorphism group.    It is easy to see that $\dimn \mathcal{I}_{\mathcal{B}}^i = \dimn \mathcal{P}_{\mathcal{B}}^{i}$.  As a consequence of Corollary \ref{3.3} and Lemma \ref{3.5}, we have:
\begin{cor}
\label{3.6}
There exists a decomposition $\alpha = \sum_l \beta^{(l)}$ into positive dimension vectors such that $\dimn (\mathcal{I}_{\mathcal{B}} - \mathcal{I}^1) = r + g \cdot \sum_{m = 1}^r \sum_{l = 1}^{s_m} (\beta^{(l)}_0)^2 - \sum_{m = 1}^r \sum_{l = 1}^{s_m} q(\beta^{(l)})$ and either $r > 1$ or $s_1 > 1$ holds.
\end{cor}

We can now state the main theorem.
\begin{thm}
\label{3.7}
Let $\alpha = \sum_l \beta^{(l)}$ be the decomposition in Corollary \ref{3.6}.  The moduli stack $\textrm{Bun}_{D,w,\alpha}(X)$ is almost good if and only if $r-1 + g \cdot \sum_{m = 1}^r \sum_{l = 1}^{s_m} (\beta^{(l)}_0)^2 - \sum_{m = 1}^r \sum_{l = 1}^{s_m} q(\beta^{(l)}) \le g \cdot \alpha_0^2 - q(\alpha)$.  It is almost very good if and only if the inequality is strict.
\end{thm}
\begin{proof}
By Theorem \ref{2.1} the almost good property is equivalent to 
$$
\dimn(\mathcal{I}_{\mathcal{B}} - \mathcal{I}^1) - 1 \le \dimn \textrm{Bun}_{D,w,\alpha}(X) =  g \cdot \alpha_0^2 - q(\alpha).
$$
Replacing the inequality with a strict one, yields the almost very good property.  The theorem follows by Corollary \ref{3.6}.
\end{proof}

\begin{rmk}
\label{3.8}
Note that in the $g = 0$ case, the almost very good property follows is equivalent to the inequality $p(\alpha) > r + \sum_l -q(\beta^{(l)})$ for some decomposition $\alpha = \sum_l \beta^{(l)}$ into at least two positive dimension vectors.  If the parabolic structure is sufficiently elaborate, this follows from the inequality $p(\alpha) > \sum_{l = 1}^r p(\gamma^{(l)})$ for some $\alpha = \sum_l \gamma^{(l)}$, as shown in \cite{So2013}.

In the $g = 1$ case, to have the almost very good property we need the inequality 
$$r-1 + \sum_{m = 1}^r \sum_{l = 1}^{s_m} (\beta^{(l)}_0)^2 - \sum_{m = 1}^r \sum_{l = 1}^{s_m} q(\beta^{(l)}) < \alpha_0^2 - q(\alpha)$$ 
to hold for a specific decomposition $\alpha = \sum_l \beta^{(l)}$ into positive dimension vectors.  We may rewrite this as: 
\begin{align*}
 0 &< 1- r + \sum_{l \neq m} 2\beta^{(l)}_0\beta^{(m)}_0 - \frac{1}{2}\sum_{l \neq m} (\beta^{(l)}, \beta^{(m)})\\ 
   &= 1-r + \sum_{l\neq m}\sum_{i,j} (\beta^{(l)}_{ij} - \beta^{(l)}_{ij+1})\beta^{(m)}_{ij+1},
\end{align*}
where $(\cdot, \cdot)$ is the symmetrization of the Euler form.  It follows that a single point with nontrivial parabolic structure in the corresponding fiber is sufficient for the moduli stack $\textrm{Bun}_{D,w,\alpha}(X)$ to be almost good.  Two points is enough for it to be almost very good.

If $g > 1$, then it is easy to see, from a similar computation, no conditions need to be placed on the parabolic structures for the almost good or almost very good properties to hold.  This is consistent with Beilinson and Drinfeld's result in \cite{BD1991}.
\end{rmk}

\section{Computations}
This section consists of several deformation theoretic computations necessary in the proof of the main theorems.  As before, all parabolic bundles we consider are over the smooth projective curve $X$.  

The following complex governs the deformation theory of $(\mathbf{W}, i)$, for a fixed $\mathbf{V}$ and inclusion of parabolic bundles $i: \mathbf{V} \hookrightarrow \mathbf{W}$:
$$
C^{\bullet}: \mathscr{E}nd_{\textrm{Par}}(\mathbf{W}) \rightarrow \mathscr{H}om_{\textrm{Par}}(\mathbf{V}, \mathbf{W}).
$$
We compute its hypercohomology.
\begin{lmm}
\label{4.3.2}
We have that $\mathbb{H}^2(X, C^{\bullet}) = 0$.
\end{lmm}
\begin{proof}
Consider the chain complexes
\begin{align*}
& A^{\bullet}: 0 \rightarrow \mathscr{E}nd_{\textrm{Par}}(\mathbf{W}) \\
& B^{\bullet}: 0 \rightarrow \mathscr{H}om_{\textrm{Par}}(\mathbf{V}, \mathbf{W}),
\end{align*}
which are nontrivial only in degree $1$.  Since $i$ induces the obvious chain map, we have an exact triangle $A^{\bullet} \rightarrow B^{\bullet} \rightarrow C^{\bullet}$, which gives rise to the long exact sequence for hypercohomology
$$
\cdots \rightarrow \mathbb{H}^2(X,A^{\bullet}) \rightarrow \mathbb{H}^2(X,B^{\bullet}) \rightarrow \mathbb{H}^2(X,C^{\bullet}) \rightarrow \mathbb{H}^3(X,A^{\bullet}) \rightarrow \cdots .
$$
Since $A^{\bullet}$ and $B^{\bullet}$ are only nontrivial in degree $1$, we have both that $\mathbb{H}^2(X,A^{\bullet}) = H^1(X, \mathscr{E}nd_{\textrm{Par}}(\mathbf{W}))$ and $\mathbb{H}^2(X,B^{\bullet}) = H^1(X, \mathscr{H}om_{\textrm{Par}}(\mathbf{V}, \mathbf{W}))$. We also obtain that $\mathbb{H}^3(X,A^{\bullet}) = 0$.  Hence, it follows that we have the exact sequence
$$
H^1(X, \mathscr{E}nd_{\textrm{Par}}(\mathbf{W})) \rightarrow H^1(X, \mathscr{H}om_{\textrm{Par}}(\mathbf{V}, \mathbf{W})) \rightarrow  \mathbb{H}^2(X,C^{\bullet}) \rightarrow 0.
$$
Therefore, it follows $\mathbb{H}^2(X,C^{\bullet})$ is the cokernel of the map induced by the inclusion $i^*: H^1(X, \mathscr{E}nd_{\textrm{Par}}(\mathbf{W})) \rightarrow H^1(X, \mathscr{H}om_{\textrm{Par}}(\mathbf{V}, \mathbf{W}))$.  Applying Serre Duality, we obtain that $\mathbb{H}^2(X,C^{\bullet})$ is isomorphic to the dual of the kernel of
$$
H^0(X, \mathscr{H}om_{\textrm{Par}}(\mathbf{W}, \mathbf{V})\otimes \Omega^1_X) \rightarrow H^0(X, \mathscr{E}nd_{\textrm{Par}}(\mathbf{W})\otimes \Omega^1_X).
$$
However, this map comes from the inclusion of $\mathscr{H}om_{\textrm{Par}}(\mathbf{W}, \mathbf{V}) \hookrightarrow \mathscr{E}nd_{\textrm{Par}}(\mathbf{W})$, which is induced by $i$.  Therefore, the map is injective, so the kernel is trivial.  Thus,  $\mathbb{H}^2(X,C^{\bullet}) = 0$.
\end{proof}

We can now compute the dimension of the stack $\mathcal{P}_{\mathbf{V}}(D,w,\alpha)$ defined in Theorem \ref{3.2}.

\begin{lmm}
\label{4.3.3}
Either $\mathcal{P}_{\mathbf{V}}(D,w,\alpha)$ is empty or we have 
$$\dimn \mathcal{P}_{\mathbf{V}}(D,w,\alpha) = \chi(\mathscr{H}om_{\textrm{Par}}(\mathbf{V}, \mathbf{W})) -  \chi (\mathscr{E}nd_{\textrm{Par}}(\mathbf{W})).$$
\end{lmm}
\begin{proof}

Assume that $\mathcal{P}_{\mathbf{V}}$ is nonempty.  The dimension of $\mathcal{P}_{\mathbf{V}}$ is equal to the dimension of the corresponding tangent complex.  We compute its dimension  by considering the deformations of $(\mathbf{W}, i) \in \mathcal{P}_{\mathbf{V}}$.  These deformations are governed by the hypercohomology of the complex $C^{\bullet}$, defined above.  It follows that
$$
\textrm{dim } \mathcal{P}_{\mathbf{V}} = \textrm{dim } \mathbb{H}^1(X, C^{\bullet}) - \textrm{dim } \mathbb{H}^0(X, C^{\bullet}),
$$
since $\mathbb{H}^2(X, C^{\bullet}) = 0$ by Lemma \ref{4.3.2}. 

Let $\chi(D^{\bullet})$ denote the Euler characteristic of the hypercohomology of a complex of sheaves $D^{\bullet}$ and let $A^{\bullet}, B^{\bullet}$ be as in Lemma \ref{4.3.2}. Since $\chi (D^{\bullet})$ additive on exact triangles, we have that
$$
\chi(C^{\bullet})  = \chi(B^{\bullet}) - \chi(A^{\bullet}).
$$
Moreover, because $\chi(B^{\bullet}) =  - \chi(\mathscr{H}om_{\textrm{Par}}(\mathbf{V}, \mathbf{W}))$ and $\chi(A^{\bullet}) = -\chi (\mathscr{E}nd_{\textrm{Par}}(\mathbf{W}))$, we can simplify this to
$$
\chi(C^{\bullet}) = \chi (\mathscr{E}nd_{\textrm{Par}}(\mathbf{W})) - \chi(\mathscr{H}om_{\textrm{Par}}(\mathbf{V}, \mathbf{W})).
$$
By Lemma \ref{4.3.2},  $\textrm{dim } \mathcal{P}_{\mathbf{V}} = - \chi(C^{\bullet})$.  Thus,
$$
\textrm{dim } \mathcal{P}_{\mathbf{V}} = \chi(\mathscr{H}om_{\textrm{Par}}(\mathbf{V}, \mathbf{W})) -  \chi (\mathscr{E}nd_{\textrm{Par}}(\mathbf{W})).
$$
\end{proof}

Let $\mathbf{F}, \mathbf{G}$ be parabolic bundles, and let $h$ be an endomorphism of $\mathbf{G}$.  Let $D^{\bullet}$ be the following chain complex: 
$$
\mathscr{H}om_{\textrm{Par}}(\mathbf{G}, \mathbf{F}) \rightarrow \mathscr{H}om_{\textrm{Par}}(\mathbf{G}, \mathbf{F}),
$$
where the connecting map is induced by $h$.
\begin{lmm}
\label{4.3.4}
We can compute the following: $\dimn \mathbb{H}^1(X, D^{\bullet}) - \dimn \mathbb{H}^0(X, D^{\bullet}) = \dimn H^1(X, \mathscr{H}om_{\textrm{Par}}(\textbf{ker } h, \mathbf{F}))$.
\end{lmm}
\begin{proof}
Since $D^{\bullet}$ consists of two copies of $\mathscr{H}om_{\textrm{Par}}(\mathbf{G}, \mathbf{F})$ we can see (by the argument from Lemma \ref{4.3.2}) that the Euler characteristic for hypercohomology is $0$.  That is, we have:
$$
\dimn \mathbb{H}^1(X, D^{\bullet}) - \dimn \mathbb{H}^0(X, D^{\bullet}) = \dimn \mathbb{H}^2(X, D^{\bullet}).
$$
By Serre duality, $\mathbb{H}^2(X, D^{\bullet})$ is isomorphic to $\mathbb{H}^0$ for the complex
$$
\mathscr{H}om_{\textrm{Par}}(\mathbf{F}, \mathbf{G}\otimes \Omega^1_{X}) \rightarrow \mathscr{H}om_{\textrm{Par}}(\mathbf{F}, \mathbf{G}\otimes \Omega^1_{X}),
$$
where the connecting map is induced by $h\otimes \text{Id}$.  However, by definition, this is just: 
$$
H^0(X, \mathscr{H}om_{\textrm{Par}}(\mathbf{F}, (\textbf{ker } h)\otimes \Omega^1_{X})) \cong H^0(X, \mathscr{H}om_{\textrm{Par}}(\mathbf{F}, \textbf{ker } h)\otimes \Omega^1_{X}).
$$
Applying Serre duality, we get: 
$$
\dimn \mathbb{H}^1(X, D^{\bullet}) - \dimn \mathbb{H}^0(X, D^{\bullet}) = \dimn \mathbb{H}^2(X,D^{\bullet}) = \dimn H^1(\mathscr{H}om_{\textrm{Par}}(\textbf{ker } h, \mathbf{F})).
$$
\end{proof}

\bibliographystyle{hacm}
\bibliography{ds}{}

\def\cydot{\leavevmode\raise.4ex\hbox{.}}
\begin{thebibliography}{1}

\bibitem{BD1991}
{\sc Beilinson, A., and Drinfeld, V.}
\newblock Quantization of {H}itchin's {I}ntegrable {S}ystem and {H}ecke
  {E}igensheaves.
\newblock \url{www.math.uchicago.edu/ mitya/langlands/hitchin/BD-hitchin.pdf},
  1991.

\bibitem{CB2003}
{\sc Crawley-Boevey, W.}
\newblock On matrices in prescribed conjugacy classes with no common invariant
  subspace and sum zero.
\newblock {\em Duke Math. J. 118}, 2 (2003), 339--352.

\bibitem{CBS2006}
{\sc Crawley-Boevey, W., and Shaw, P.}
\newblock Multiplicative preprojective algebras, middle convolution and the
  {D}eligne-{S}impson problem.
\newblock {\em Adv. Math. 201}, 1 (2006), 180--208.

\bibitem{Ko22004}
{\sc Kostov, V.~P.}
\newblock The {D}eligne-{S}impson problem---a survey.
\newblock {\em J. Algebra 281}, 1 (2004), 83--108.

\bibitem{La1988}
{\sc Laumon, G.}
\newblock Un analogue global du c\^one nilpotent.
\newblock {\em Duke Math. J. 57}, 2 (1988), 647--671.

\bibitem{MS1980}
{\sc Mehta, V.~B., and Seshadri, C.~S.}
\newblock Moduli of vector bundles on curves with parabolic structures.
\newblock {\em Math. Ann. 248}, 3 (1980), 205--239.

\bibitem{Ra2009}
{\sc Raskin, S.}
\newblock The {C}otangent {S}tack.
\newblock
  \url{http://www.math.harvard.edu/~gaitsgde/grad_2009/SeminarNotes/Sept22(Dmodstack1).pdf},
  2009.

\bibitem{Si1991}
{\sc Simpson, C.~T.}
\newblock Products of matrices.
\newblock In {\em Differential geometry, global analysis, and topology
  ({H}alifax, {NS}, 1990)}, vol.~12 of {\em CMS Conf. Proc.} Amer. Math. Soc.,
  Providence, RI, 1991, pp.~157--185.

\bibitem{So2013}
{\sc Soibelman, A.}
\newblock The moduli stack of parabolic bundles over the projective line,
  quiver representations, and the {D}eligne-{S}impson problem.
\newblock arXiv:1310.1144.

\end{thebibliography}

\end{document}